\documentclass{amsart}
\usepackage{amsmath,amssymb,enumerate,mathrsfs}
\newtheorem{theorem}{Theorem}[section]
\newtheorem{lemma}[theorem]{Lemma}

\newtheorem*{thmTT}{Theorem TT}

\newtheorem*{prop1}{Proposition 1}
\newtheorem*{thm1}{Theorem 1}
\newtheorem*{thm2}{Theorem 2}

\newtheorem{example}[theorem]{Example}

\theoremstyle{definition}
\newtheorem*{definition}{Definition}
\newtheorem{remark}[theorem]{Remark}
\newtheorem{problem}[theorem]{Problem}
\newtheorem*{problemTT}{Problem TT}
\newtheorem*{problemLebNull}{Problem}

\def\en{\mathbb N}
\def\er{\mathbb R}

\def\H{\mathcal H}

\def\A{\mathcal A}

\def\I{\mathscr I}

\newcommand{\graph}{\operatorname{graph}}

\newcommand{\lin}{\operatorname{Lin}}

\newcommand{\eps}{\varepsilon}

\def\eqn#1$$#2$${\begin{equation}\label#1#2\end{equation}}

\begin{document}

\author{Du\v{s}an Pokorn\'y, Martin Rmoutil}
\title{On Removable Sets For Convex Functions}
\thanks{D.~Pokorny was supported by a cooperation grant of the Czech and the German science foundation, GA\v CR project no.\ P201/10/J039, M.~Rmoutil by the Grant No. 710812 of the Grant Agency of the Charles University in Prague. }
\email{dpokorny@karlin.mff.cuni.cz, caj@rmail.cz}
\address{Charles University, Faculty of Mathematics and Physics, Sokolovsk\'a 83, 186 75 Praha 8 Karl\'{\i}n, Czech Republic}
\subjclass[2010]{26B25, 52A20}
\keywords{Convex function, locally convex function, intervally thin set, $c$-removable set, convex extension, separately convex function}
\begin{abstract}
In the present article we provide a sufficient condition for a closed set $F\in \er^d$ to have the following property which we call $c$-removability: Whenever a function $f:\er^d\to \er$ is locally convex on the complement of $F$, it is convex on the whole $\er^d$. We prove that no generalized rectangle of positive Lebesgue measure in $\er^2$ is $c$-removable. Our results also answer the following question asked in an article by Jacek Tabor and J\'ozef Tabor [{\it J. Math. Anal. Appl.} {\bf 365} (2010)]: Assume the closed set $F\subset \er^d$ is such that any locally convex function defined on $\er^d\setminus F$ has a unique convex extension on $\er^d$. Is $F$ necessarily intervally thin (a notion of smallness of sets defined by their ``essential transparency'' in every direction)? We prove the answer is negative by finding a counterexample in $\er^2$. 
\end{abstract}
\maketitle

\section{Introduction}
The present article is mostly motivated by the work \cite{Taborite} about negligible sets for convexity of functions in $\er^d$, where an interesting open problem was raised. We shall need the following notion introduced in \cite{Taborite}.

A set $A\subset \er^d$ is called \emph{intervally thin} if for any $x,y\in \er^d$ and any $\eps >0$ there exist $x'\in B(x,\eps)$ and $y'\in B(y,\eps)$ such that $[x',y']\cap A=\emptyset$.

\begin{problemTT}
Let $A\subset \er^n$ be closed. Suppose that for an arbitrary open set $U$ containing $A$ every locally convex function $f:U\setminus A\to\er$ has a unique extension on $U$. Is it then necessarily true that $A$ is intervally thin?
\end{problemTT}

Arguably our main result is that the answer to this question is negative. Example~\ref{HDS} and Remark~\ref{rem} provide a closed set $K$ which is not intervally thin, but which enjoys the ``unique extension property for convex functions'' (UEP) from Problem TT. We took the liberty of calling this set $K$ ``the Holey Devil's Staircase'' since it is the graph of the classical Cantor function (the Devil's Staircase) minus all the horizontal open line segments contained in the graph (in other words, it is the graph of the restriction of the Cantor function to the Cantor set).

One can readily verify that the Holey Devil's Staircase is not intervally thin. It is enough to consider the last intersection of the graph of the Cantor function with any line segment with endpoints in $(-\infty,0)\times\left(0,\frac12\right)$ and $(1,\infty)\times\left(\frac12,1\right)$; clearly, this intersection is an element of $K$.

To prove that $K$ has the UEP, is considerably more difficult and our effort in this direction has inspired a large part of this article. 

The main result of \cite{Taborite} is essentially the following theorem. Note that since we restrict our attention to convex functions (as opposed to $\omega$-semiconvex functions studied in \cite{Taborite}), we change the formulation of the theorem accordingly:

\begin{thmTT}
Let $U$ be an open subset of $\er^d$ and let $A$ be a closed intervally thin subset of $U$. Let $f:U\setminus A\to \er$ be a locally convex function. Then $f$ has a unique locally convex extension on $U$.
\end{thmTT}

\noindent The proof of this theorem consists of two principal steps: 
\begin{enumerate}[(1)]
\item First, one proves that there is a unique continuous extension; this is the more difficult part.
\item Once one has the continuous extension, it is then easy to prove that it is convex. 
\end{enumerate}

Our aim is to apply this scheme to our set $K$. It turns out that in this case the easier step is (1); we only need a simple generalization of the corresponding theorem from \cite{Taborite}---which we have in Lemma \ref{ext}. 

Performing step (2) for $K$ is the crucial part and it motivates the introduction of $c$-removable sets with the consequent natural question: Which sets are $c$-removable?

\begin{definition}
We say that a closed set $A\subset \er^d$ is \emph{$c$-removable} if the following is true: Every real function $f$ on $\er^d$ is convex whenever it is continuous on $\er^d$ and locally convex on $\er^d\setminus A$.
\end{definition}

A consequence of Theorem TT is that all closed intervally thin sets are $c$-removable, but this fact does not help us. In $\er^2$ we were able to find a sufficient condition more general than interval thinness which covers also the case of our set~$K$:

\begin{prop1}\label{prop1}
Let $K\subset\er^2$ be compact and intervally thin in two different directions. 
Assume that for a dense set of line segments $L\subset \er^2$ the cardinality of
$K \cap L$ is at most countable. Then K is $c$-removable.
\end{prop1}
Here interval thinness of $K$ in a direction means that to any given line segment in that direction we can find arbitrarily close line segments contained in the complement of the set $K$. It is not difficult to see that for any closed set $K$ intervally thin in a direction $v$, any continuous function which is locally convex outside $K$ is necessarily convex on all lines parallel to $v$. Hence, the assumption of interval thinness of $K$ in two directions ensures that our function is convex in those two directions (i.e. is essentially separately convex) which we can use further in the proof---the key Lemma~\ref{sverak} tells us that a separately convex function cannot ``have a concave angle'' on any line.

The condition from the proposition may seem rather artificial, but it emerges quite naturally from our method of the proof. What is more, it is easily seen to be more general than interval thinness and is fulfilled by $K$. (Hence, the Holey Devil's Staircase is $c$-removable.) However, we were not able to generalize this condition to higher dimensions; instead, we used the geometric measure theory to obtain the following theorem which in $\er^2$ is strictly weaker than Proposition~1.

\begin{thm1}\label{theorem1}
Let $M\subset\er^d$ be a compact set which is intervally thin in $d$ linearly independent directions $n_1,\dots,n_d$.
Suppose that $M$ has $\sigma$-finite $(d-1)$-dimensional Hausdorff measure. Then $M$ is $c$-removable. 
\end{thm1}

This condition does not include interval thinness because there are intervally thin sets of positive $d$-dimensional measure in $\er^d$. For instance, in $\er^2$ one can construct such a set by taking the full unit square and digging in it countably many straight tunnels in such a way that the rest is intervally thin but still of positive measure.

Among other signs, also from this fact it seems rather obvious that this theorem is far from being a characterization of $c$-removable sets. In fact, it is not even clear whether all $c$-removable sets in $\er^2$ are totally disconnected; from the considerations contained in the second part of Section~5 it seems plausible that the Koch curve might be an example of a non-trivial $c$-removable continuum in $\er^2$. (Of course, such an example has to be rather complicated as it is not difficult to prove that no smooth curve in $\er^2$ is $c$-removable.)

On the other hand, we have the following.

\begin{thm2}\label{theorem2}
Let $A, B\subseteq\er$ be closed sets of positive Lebesgue measure. Then $A\times B$ is not $c$-removable.
\end{thm2}
This theorem is interesting only for $A$, $B$ totally disconnected (otherwise $A\times B$ contains a non-degenerated line segment and the statement is trivial). 
However, we do not know (and would like to know) whether e.g. the Cantor dust ($C\times C$ where $C$ is the Cantor set) is $c$-removable. As a matter of fact, possibly the most interesting of related open problems is:

\begin{problemLebNull}
Is there a closed totally disconnected Lebesgue null set in $\er^2$ which is not $c$-removable?
\end{problemLebNull}

It is worth pointing out that Theorem~2 is related to the recent work \cite{P} where a totally disconnected compact set which is not $c$-removable is constructed. The construction is rather complicated, but the witnessing function has a compact support, making the example stronger. However, even Theorem~2 is enough to achieve the main goal of \cite{P}, which is to disprove a theorem by L.~Pasqualini from 1938 \cite[Theorem~51]{Pasqualini} stating that any totally disconnected compact set in $\er^2$ is $c$-removable. It was the connection to this old article what convinced us that, of the two steps involved in the proof of Theorem~TT,  the crucial one is actually the second.

\section{Notation and basic facts}
All spaces shall be equipped with the Euclidean metric. We denote by $B(x,\eps)$ the open ball (with respect to the Euclidean metric) with the centre $x$ and radius $\eps$. Since confusion is unlikely, the symbol $(x,y)$ denotes an open interval in $\er$ as well as the point in $\er^2$ with coordinates $x$ and $y$. Similarly the symbol $[x,y]$ may denote a closed interval (when $x,y\in \er$) as well as the line segment with endpoints $x$ and $y$ (when $x,y\in \er^d$, $d>1$). By $\H^{k}$ we denote the $k$-dimensional Hausdorff measure. 
For $M\subset\er^d$ and $\alpha$ a countable ordinal we denote the $\alpha$-th Cantor-Bendixson derivative of $M$ by $M^{(\alpha)}$.
The unit sphere in $\er^d$ is denoted by $S^{d-1}$. 
For $v\in\er^d$ we denote the orthogonal complement of $v$ by $v^\bot$.
The symbol $\lin M$ denotes the linear span of $M\subset\er^d$.
For a fixed $d\in\en$ denote the standard basis of $\er^d$ by $\{e_1,\dots,e_d\}$.

Let $U\subset\er^d$ be open and $f:U\to\er$ be a function. We say that $f$ is \emph{locally convex} on $U$ if for some open convex $V\subset U$ the function $f|_V$ is convex. It is easy to see that a locally convex function is convex on any convex set contained in its domain.

The set $A\subset\er^d$ is called \emph{$k$-rectifiable} if there exist countably many Lipschitz mappings $f_i:\er^{k} \to \er^d$
such that 
$$
\H^{k}\bigg(A\setminus \bigcup_{i=0}^\infty f_i\left(\er^{k}\right)\!\bigg)=0.
$$
Since we will work only with the case $k=d-1$, we will call $(d-1)$-rectifiable sets just \emph{rectifiable}.

Let $G(d,k)$ be the Grassmannian of $k$-dimensional linear subspaces of $\er^d$ equipped with the unique invariant probability measure $\nu_k^d$.
Besides the Hausdorff measure we will also use the $k$-dimensional Favard measure (integralgeometric measure) $\I^k$ on $\er^d$ which is for a Borel set $M$ defined as
$$
\I^k(M)=\frac{1}{\beta(d,k)}\int_{G(d,k)} \int _{V} \H^0(M\cap p_V^{-1}(y))\;d\H^{k}(y)\;d\nu_k^d(V),
$$
where $p_V$ is the orthogonal projection to $V$ and the number $\beta(d,k)$ is a non-zero constant depending only on $d$ and $k$ whose precise value is not important for us.

We will also need the following properties of the Favard measures. 
Let $M\subset\er^d$ be a Borel set such that $\H^{d-1}(M)<\infty.$
Then $M$ can be expressed as a union of a rectifiable set $R$ and a set $P$ satisfying $\I^{d-1}(P)=0$ (c.f. \cite[3.3.13]{F}).
Moreover, each rectifiable set $R\subset\er^d$ satisfies $\I^{d-1}(R)=\H^{d-1}(R)$ (c.f. \cite[3.2.26]{F}).

\section{Separately convex functions}
The following lemma is a variant of an unpublished observation by V. \v Sver\'ak (see \cite{T}).
 For the convenience of the reader we provide a proof as we were not able to find one in the literature.
\begin{lemma}\label{sverak}
Let $f:\er^2\to\er$ be a separately convex function. Define $g:\er\to\er$ by $g(t)=f(t,t).$
Then 
$$
 \liminf_{t\to 0+} \frac{g(x+t)+g(x-t)-2g(x)}{t}\geq 0
$$ 
for every $x.$
\end{lemma}

\begin{proof}
Without any loss of generality we can suppose that $x=(0,0)$, $f(z)=0$ and that there is a sequence $t_n\searrow0$ such that for each $n\in\en$,
\begin{equation}\label{direct}
 \frac{g(t_n)+g(-t_n)}{t_n}\leq -1.
\end{equation}
For $t>0$ put
$$
\sigma (t):=f(t,-t)+f(-t,t)
\quad\text{and}\quad
\rho (t):=f(t,t)+f(-t,-t).
$$
Note that, since $f$ is separately convex for every $t$,
\begin{equation}\label{zero}
\sigma(t)+\rho(t)\geq 0.
\end{equation}
Now we shall prove the following claim.

\emph{If for some $t_0$ and some $p$,
$$
\sigma(t_n)\geq pt_n \quad\text{for every $n$,}
$$
then
$$
\sigma(t_n)\geq (p+2)t_n \quad\text{for every $n$}.
$$
}
This is enough to prove the lemma since \eqref{zero} and \eqref{direct}, together with the above claim, imply $\sigma(t_n)\geq L$ for every $n$ and every $L\in\er$, which is not possible.

To prove the claim, first observe that due to \eqref{direct} we know that for each $n$,
$$
\rho(t_n)\leq -t_n.
$$
This implies for each $n$,
$$
\frac{\sigma(t_n)-\rho(t_n)}{2t_n}\geq \frac{p+1}{2}.
$$
By separate convexity of $f$, we get that
$$
f(t_k,t_n)+f(-t_k,-t_n)\geq \sigma(t_k)+ (t_n-t_k)\cdot\frac{\sigma(t_k)-\rho(t_k)}{2t_k} 
$$
provided $k>n.$
If we consider $k\to\infty$, we obtain
$$
f(0,t_n)+f(0,-t_n)\geq t_n\frac{p+1}{2}.
$$
Using the separate convexity of $f$ one more time together with \eqref{direct} we get
\begin{equation*}
\sigma(t_n)\geq f(0,t_n)+f(0,-t_n)+(f(0,t_n)+f(0,-t_n)-\rho(t_n))\geq (p+2)t_n.
\qedhere
\end{equation*}
\end{proof}


\begin{definition}
The set $A\subset \er^d$ is called \emph{intervally thin in direction $v\in S^{d-1}$} if for any $x,y\in \er^d$ with $x-y$ parallel to $v$, and any $\eps >0$, there exist $x'\in B(x,\eps)$ and $y'\in B(y,\eps)$ such that $[x',y']\cap A=\emptyset$.
\end{definition}


\begin{proof}[Proof of Theorem~1]
First we claim that for every non-convex function $f$, there is a line $L$ such that $L\cap M$ is countable and $f|_L$ is non-convex. Moreover, $L$ can be found such that $v\in\lin\{n_i: i\in J\}$
for no $J\subsetneq\{1,\dots,d\}$, where $v$ is the direction of $L.$

To prove the claim first express $M$ as a countable union of sets $M_n$ satisfying $\H^{d-1}(M_n)<\infty.$
By \cite[3.3.13]{F} we can express each $M_n$ in the form $P_n\cup R_n$ with $\I^{d-1}(P_n)=0$ and
$R_n$ rectifiable. 

Fix $n\in\en.$ Using \cite[3.2.26]{F} we see that 
$$
\I^{d-1}(R_n)=\H^{d-1}(R_n)<\infty.
$$ 
This means, by the definition of the Favard measure, that for almost every $H\in G(d,d-1)$, almost every line perpendicular to $H$ intersects $R_n$ in at most finitely many points.
In particular, this means that almost every line intersects $M$ in finitely many points.

So, putting $P=\bigcup P_n$ and $R=\bigcup R_n$, we have that $\I^{d-1}(P)=0$, and also that almost every line intersects $R$ in at most countably many points.
Hence, almost every line intersects $M=P\cup R$ in at most countably many points.

Since $f$ is non-convex, the set $\A$ consisting of all lines $L$ such that $f|_L$ is non-convex has a positive measure.
To finish the proof of the claim we simply pick a line from $\A$ such that $L\cap M$ is at most countable.

Fix $z\in L$. 
Due to the last part of the claim, we can suppose (possibly by composing $f$ with a suitable affine mapping) that $v_i=e_i$,
$v=\frac{1}{\sqrt d}(1,1,\dots,1)$  and such that $z=(0,0,\dots,0)$.

Now, we will prove the statement of the theorem using induction on $d$.
Suppose that $d=2$ and that $f$ is locally convex on $M^c$.
Put $K:=M\cap L$; then $K$ is a countable compact.
By Lemma~\ref{sverak} we know that if $f|_L$ is convex on $N^c$ for some $N$ then it is convex on a neighbourhood of any isolated point $N$.
This means that $f_L$ is convex on $K ^{(\alpha)}$ for every countable ordinal $\alpha.$
But, since $K$ is a countable compact, there is a countable ordinal $\beta$ such that $K^{(\beta)}=\emptyset$ which is a contradiction with the assumption of $f$ being non-convex.
This finishes the proof for $d=2.$

Suppose that the lemma is true for every $d$ up to $k-1\geq 2$, we will prove that it is true for $d=k$ as well.
Put $\nu:=v-v\cdot e_1$, $A(p):=pe_1+\lin\{e_2,\dots,e_d\}$ and $L(p):=pe_1+\lin\{\nu\}$ for $p\in\er.$
Then it is easy to verify that one of the following two statements is true:
\begin{enumerate}
\item[(a)] $f|_{L(p)}$ is non-convex for every $p$ from some interval $(a,b)$,
\item[(b)] $f|_{L(p)}$ is convex for every $p$.
\end{enumerate}
If $(a)$ is valid than by \cite[Theorem~7.7]{Mat} we know that $M\cap A(q)$ is of $\sigma$-finite $(d-2)$-dimensional Hausdorff measure for some $q\in(a,b)$.
But this is not possible by applying induction procedure to the function $f|_{A(q)}.$

One the other hand, $(b)$ is not possible either.
Indeed, we can apply Lemma~\ref{sverak} to $f|_{\lin\{v,e_1\}}$ the same way as in the proof of the case $d=2$.
\end{proof}

Note that the proof for $d=2$ directly gives us Proposition~\ref{prop1}.

\section{Extensions of locally convex functions}

\begin{definition}
We say that a set $A\subset\er^{d}$ is \emph{totally disconnected in a direction $v\in S^{d-1}$} if the set $A\cap l$ is totally disconnected for every line $l$ parallel to $v.$

\end{definition}

The following lemma is a refinement of \cite[Theorem 3.1]{Taborite} (note that the non-trivial part of the theorem is the existence of a unique continuous extension).

\begin{lemma}\label{ext}
Suppose that $A\subset\er^{d}$ is closed and both totally disconnected and intervally thin in some direction $v\in S^{d-1}$.
Let $U\subset\er^{d}$ be open.
Then every function locally convex on $A^{c}\cap U$ admits a continuous extension to $U.$
\end{lemma}

\begin{proof}
Let $f:U\to\er$ be locally convex on $A^{c}\cap U$.
Choose $x\in U$ and $\eps>0$ we need prove that there is a $\delta>0$ such that if $|x-a|,|x-b|\leq\delta$ and $a,b\in U\cap A^{c}$ then $|f(a)-f(b)|<\eps.$
Without any loss of generality we can suppose that $x=0$ and that $v$ is parallel to one of the coordinate axis.

For $u\in\er^{d}$ and $r>0$ put $l_u:=u+\lin\{v\}$ and $C(u,r):=u+[-r,r]^d$.
Since $A$ is totally disconnected in the direction $v$ we can find $\alpha>0$ and $\frac{\alpha}{2}>\gamma>0$ such that for $y:=\alpha v$ we have $C(y,\gamma),C(-y,\gamma)\subset A^{c}\cap U.$
Since $f$ is locally convex on $A^{c}$ and therefore locally Lipschitz on $A^{c}$, there is $K>0$ such that $f$ is K-Lipschitz on both $C(y,\gamma)$ and $C(-y,\gamma).$
Using the fact that $A$ is totally disconnected in the direction $v$ again, we can find $\min\left(\frac{\eps}{25K},\alpha-2\gamma\right)>\lambda>0$ such that for $z:=\lambda v$ we have $z\in A^{c}$.
Since $f$ is continuous on $A^{c}$, there is $\lambda>\delta>0$ such that for every $u\in C(z,\delta)$ we have $|f(z)-f(u)|\leq\frac{\eps}{4}$.

To obtain a contradiction, suppose that there are $a,b\in C(x,\delta)\cap A^{c}$ such that $|f(a)-f(b)|\geq\eps.$
Let $x_a$ and $x_b$ be the unique point in $(z+v^{\bot})\cap l_a$ and $(z+v^{\bot})\cap l_b,$ respectively.
Then $x_a,x_b\in C(z,\delta)$ and so $|f(z)-f(x_a)|\leq\frac{\eps}{4}$ and $|f(z)-f(x_b)|\leq\frac{\eps}{4}.$ 
Moreover, one of the inequalities 
$$
f(a)-f(z)\geq\frac{\eps}{2},\;\;f(a)-f(z)\leq-\frac{\eps}{2},\;\; f(b)-f(z)\geq\frac{\eps}{2},\;\;f(b)-f(z)\leq-\frac{\eps}{2}
$$ 
must hold.
Therefore, one of the inequalities 
\begin{equation}\label{nerovnice}
f(a)-f(x_a)\geq\frac{\eps}{4},\;\;f(a)-f(x_a)\leq-\frac{\eps}{4},\;\;f(b)-f(x_b)\geq\frac{\eps}{4},\;\;f(b)-f(x_b)\leq-\frac{\eps}{4}
\end{equation}
must hold as well.

Now, consider for instance the inequality $f(a)-f(x_a)\geq\frac{\eps}{4}$.
Since $A$ is intervally thin in the direction $v$ there are three colinear points $s_y\in C(-y,\gamma)$, $s_a\in C(x,\delta)$ and $s_x\in C(z,\delta)$ such that $[s_y,s_x]\subset U\setminus A$ and such that
$$
|f(s_a)-f(a)|, |f(s_x)-f(x_a)|\leq\frac{\eps}{16}.
$$
Then we have
\begin{equation}\label{eska}
|s_a-s_x|\leq |x-y|+2\delta=\lambda+2\delta\leq3\lambda\leq\frac{3\eps}{25K}.
\end{equation}
Moreover,
\begin{equation}\label{efka}
f(s_a)-f(s_x)\geq f(a)-f(x_a)-|f(s_x)-f(x_a)|-|f(s_a)-f(a)|\geq \frac{\eps}{4}-\frac{\eps}{16}-\frac{\eps}{16}=\frac{\eps}{8}.
\end{equation}
Using (\ref{eska}) and(\ref{efka}) we obtain
\begin{equation}\label{sklon}
\frac{f(s_a)-f(s_x)}{|s_a-s_x|}\geq \frac{\eps}{8}\cdot\frac{25K}{3\eps}>K.
\end{equation}
Choose an arbitrary $w\in([s_y,s_x]\setminus\{s_y\})\cap C(-y,\gamma)$.
From the convexity of $f$ on $[s_y,s_a]$, the fact that $[s_y,w]\subset[s_y,s_x]\subset[s_y,s_a]$ and \eqref{sklon} we obtain
$$
K<\frac{f(s_a)-f(s_x)}{|s_a-s_x|}\leq \frac{f(s_y)-f(w)}{|s_y-w|}\leq K,
$$
which is not possible. The remaining cases in \eqref{nerovnice} can be proved following the same lines.
\end{proof}

\begin{example}\label{HDS}
There is a compact set $K\subset\er^2$ which is not intervally thin and such that for every $f:K^{c}\to\er$ locally convex on $K^c$ there is a convex extension $F:\er^2\to\er.$ 
\end{example}
\begin{proof}
Let $h : [0, 1]\to [0, 1]$ be the classical Cantor function (the
Devil's Staircase) and let $C \subset[0,1]$ be the Cantor ternary set.
Now define the set $K$ as the graph of $h$ restricted to $C$.

First note that $\H^1(K)<\H^1(\graph h)<\infty$ and that $K$ is intervally thin in directions $(1,0)$ and $(0,1).$
Therefore using Theorem~\ref{theorem1} and Lemma~\ref{ext} we obtain that $K$ has the desired extension property and so it remains to prove that $K$ is not intervally thin.
Define $H:\er\to\er$ by $H=h$ on $[0,1]$, $H=0$ on $(-\infty,0)$ and $H=1$ on $(1,\infty).$ 
Then $\er^2\setminus\graph H$ has two components, say $C^+$ and $C^-$, and therefore for any $x^\pm\in C^\pm$ the line segment $[x^+,x^-]$ intersects $\graph H.$
Now take $x^+\in B((-\frac{1}{3},\frac{1}{3}),\frac{1}{4})$ and $x^-\in B((\frac{4}{3},\frac{2}{3}),\frac{1}{4})$ and set
$$
x:=\sup\{a\in\er: \text{there exists $b\in\er$ such that $(a,b)\in[x^+,x^-]\cap \graph H$}\}.
$$
Then $(x,H(x))\in K.$
\end{proof}

\begin{remark}\label{rem}
Note that the set $K$ from Example~\ref{HDS} also provides an answer to Problem~TT (see the introduction) in the negative. 
This can be seen from the fact that the argument used in the proof of Example~\ref{HDS} can be easily localized using the  self affinity of $K$.
\end{remark}

\section{Two examples}
In the first part of this section we prove Theorem~2 which provides us with a very natural class of examples of sets which are not $c$-removable. We shall need the following definitions.

Let $\beta,\eps>0$. Then we define the function $g_{\beta,\eps}:\er^2\to\er$ by
$$
g_{\beta,\eps}(x,y):=
\begin{cases}
\beta y^2-2\eps x-\eps^2,\ &(x,y)\in(-\infty,-\eps]\times\er,\\
\beta y^2+x^2,\ &(x,y)\in(-\eps,\eps)\times\er,\\
\beta y^2+2\eps x-\eps^2,\ &(x,y)\in[\eps,\infty)\times\er.
\end{cases}
$$
For $w\in\er$, set 
$$g_{\beta,\eps}^{w}(x,y):=g_{\beta,\eps}(x-w,y) \quad\text{and}\quad h_{\beta,\eps}^{w}(x,y):=g_{\beta,\eps}^{w}(y,x).$$
One can readily verify that all the functions just defined are convex and $C^1$.

The Hessian matrix of $g_{\beta,\eps}^{w}$ in $(-\infty,-\eps+w)\times\er$, $(-\eps+w,\eps+w)\times\er$ and $(\eps+w,\infty)\times\er$, respectively, is
\begin{equation}\label{hessvert}
\left(\begin{array}{cc}
	0 & 0\\
	0 & 2\beta
\end{array}\right),
\quad
\left(\begin{array}{cc}
	2 & 0\\
	0 & 2\beta
\end{array}\right)
\quad\text{and}\quad
\left(\begin{array}{cc}
	0 & 0\\
	0 & 2\beta
\end{array}\right),
\end{equation}
and similarly the Hessian matrix of $h_{\beta,\eps}^{w}$ in $\er\times(-\infty,-\eps+w)$, $\er\times(-\eps+w,\eps+w)$ and $\er\times(\eps+w,\infty)$, respectively, is

\begin{equation}\label{hesshor}
\left(\begin{array}{cc}
	2\beta & 0\\
	0 & 0
\end{array}\right),
\quad
\left(\begin{array}{cc}
	2\beta & 0\\
	0 & 2
\end{array}\right)
\quad\text{and}\quad
\left(\begin{array}{cc}
	2\beta & 0\\
	0 & 0
\end{array}\right).
\end{equation}

Further, define the functions $f_i:\er^2\to \er$, $i\in\{1,2,3,4\}$, as follows:
$$
f_1(x,y):=
\begin{cases}
\frac{1}{12}x^2+4(y-1)^2,\ &(x,y)\in\er\times (1,\infty),\\
\frac{1}{12}x^2,\ &(x,y)\in\er\times (-\infty, 1],
\end{cases}
$$
and $f_2(x,y):=f_1(x,-y)$, $f_3(x,y):=f_1(y,x)$ and $f_4(x,y):=f_3(-x,y)$.

Again, it is easy to check that the functions $f_i$, $i=1,2,3,4$ are $C^1$ and that the Hessian matrix of (e.g.) $f_1$ in $\er \times (1,\infty) $ and $\er \times(-\infty,1)$, respectively, is
\begin{equation}\label{hessvenku}
\left(\begin{array}{cc}
	\frac{1}{6} & 0\\
	0 					& 8
\end{array}\right)
\quad\text{and}\quad
\left(\begin{array}{cc}
	\frac{1}{6} & 0\\
	0 					& 0
\end{array}\right).
\end{equation}
It is also useful to note that the Hessian Matrix of $\varphi:(x,y)\mapsto -xy$ is
\begin{equation}\label{hessnekonv}
\left(\begin{array}{rr}
	 0 & -1\\
	-1 &  0
\end{array}\right).
\end{equation}

\begin{lemma}\label{LCRem}
Let there, for each $i\in\en$, be given $\eps_i>0$, $\beta_i \in \left(0, \frac{1}{80}\right)$ and $w_i\in(-1,1)$ such that $\sum\beta_i=\frac{1}{4}$, $\sum\eps_i<\frac{1}{24}$ and $(w_i-\eps_i,w_i+\eps_i)$ are pairwise disjoint intervals contained in $[-1,1]$. 
Denote $g_i:=g_{\beta_i,\eps_i}^{w_i}$ and $h_i:=h_{\beta_i,\eps_i}^{w_i}$, define the function $f:\er^2\to\er$ as
$$
f(x,y):=\varphi(x,y)+\sum_{i=1}^4 f_i(x,y)+ \sum_{i=1}^\infty\left(g_i(x,y) +h_i(x,y)\right),
$$
and define the set $K\subset \er$ as
$$
K:=[-1,1]\setminus \bigcup_{i=1}^\infty (w_i-\eps_i,w_i+\eps_i).
$$
Then $f$ has the following properties:
\begin{enumerate}
\item It is non-convex, since 
$$\frac{f(-1,-1)+f(1,1)}{2}<f(0,0);$$
\item it is locally convex on $\er^2 \setminus K^2$.
\end{enumerate}
\end{lemma}

\begin{proof}
First, we need to check that $f$ is a well-defined function. To that end, it is sufficient to note that for any $R>1$ and any $i\in\en$ the maximum of $g_i$ on $[-R,R]^2$ is attained at the point $(R,R)$ and is less than $\beta_i R^2+2\eps_i 2R$. Obviously, the same is true for $h_i$, and so the infinite series in the definition of $f$ converges locally uniformly on $\er^2$.

Now, the Hessian matrix of $f$ at a point $(x,y) \in \er^2\setminus K^2$ need not exist; however, all the summands in the definition of $f$, excluding $\varphi$, are convex functions. Since the sum of convex functions is convex, it is enough to prove that at any point $(x,y)$ outside $K^2$ we can find finitely many $f_i$'s, $g_i$'s and $h_i$'s such that their sum together with $-xy$ has a positively definite Hessian matrix at $(x,y)$.

This is easy to do if $(x,y)\notin[-1,1]^2$; in this case we only need to use one of the $f_i$'s. For example, if $(x,y)\in \er\times (1,\infty)$, then we see from \eqref{hessvenku} and \eqref{hessnekonv} that $\varphi+f_1$ is convex at $(x,y)$.

The ``worst case'' is that $(x,y)\in [-1,1]^2$ lies in a single vertical stripe of the form $(w_k-\eps_k,w_k+\eps_k)\times\er$ (fix the $k$) and not in any of the horizontal stripes of the form $\er\times (w_i-\eps_i, w_i+\eps_i)$. 
In this case we find a finite set $F\subset\en$ such that $(x,y)$ does not lie on the boundary of any stripe of the form $\er\times (w_i-\eps_i, w_i+\eps_i)$ with $i\in F$ (which can happen at most twice) and such that
$$
\alpha_F:=\sum_{i\in F} \beta_i > \frac{9}{40}\,.
$$
Let us now consider the function
$$
f_F:=\varphi+\sum_{i\in F}\left(g_i+h_i \right) +
\begin{cases}
g_k, \ & \text{if } k\notin F; \\
0,   \ & \text{if } k\in F.
\end{cases}
$$
Since we know that $(x,y)$ does not lie on the boundary of any of the stripes (horizontal or vertical) involved in the definition of $f_F$, the Hessian matrix of $f_F$ exists on a convex open neighbourhood $U$ of $(x,y)$ and it is easy to see from \eqref{hessvert} and \eqref{hesshor} that its determinant satisfies
$$
\det\left(H_{f_F}(a,b)\right)\geq 4\left(\alpha_F+\alpha_F^2\right)-1>4\left(\frac{9}{40}+\frac{81}{1600}\right)-1>0, \quad (a,b)\in U.
$$
As $\frac{\partial^2 f_F}{\partial x^2}>0$, we obtain that $f_F$ is convex on $U$. 
On the other hand, $f-f_F$ is convex and therefore $f$ is convex on $U$; one can check the other cases in a similar way concluding the proof of property $(b)$. 

It remains to verify property $(a).$
Denote
$$
\delta_i:=\frac{g_i(-1,-1)+g_i(1,1)}{2}-g_i(0,0) .
$$
An easy computation shows that
$$
\delta_i=\beta_i+2\eps_i(1-|w_i|)
$$
and clearly the same is also true if we substitute all the occurrences of $g$ in the definition of $\delta_i$ by $h$. Hence, 
$$
\sum_{i=1}^\infty \delta_i=\sum_{i=1}^\infty \beta_i+2\sum_{i=1}^\infty(\eps_i(1-|w_i|))<\frac{1}{4}+2\sum_{i=1}^\infty\eps_i<\frac{1}{3}\,.
$$
We also have
$$
\sum_{i=1}^4 \left(\frac{f_i(-1,-1)+f_i(1,1)}{2}-f_i(0,0)\right)=4\cdot\frac{1}{12}=\frac{1}{3}\,.
$$
The last two facts clearly imply $(a)$.
\end{proof}

\begin{proof}[Proof of Theorem~2]
If $A\times B$ contains a line segment then it is not $c$-removable as follows from \cite[Example 2.1]{Taborite}. From now on, assume that the sets $A$ and $B$ are totally disconnected.

By the Lebesgue density theorem we can find points $a\in \er$ and $b\in \er$ which are density points of $A$ and $B$ respectively. Without loss of generality assume that $a=b=0$; since $A$ and $B$ are both closed, $0\in A\cap B$. Then $0$ is clearly a point of density of $A\cap B\cap (-A)\cap(-B)$; consequently we can further assume that the sets $A$ and $B$ are symmetrical.  We shall prove that $(A\cap B)^2$ is not $c$-removable.

Take an $r>0$ such that $\frac{1}{2r}\lambda\left(A\cap B \cap (-r,r) \right) > \frac{23}{24}$ and such that $r$ (and therefore also $-r$) is in $A\cap B$. Without loss of generality we can assume that $r=1$. Now, for $i \in \en$ take $w_i\in(-1,1)$ and $\eps_i>0$ such that the intervals $(w_i-\eps_i,w_i+\eps_i)$ are pairwise disjoint and such that 
$$[0,1]\setminus A\cap B= \bigcup_{i=1}^\infty(w_i-\eps_i, w_i+\eps_i).$$
The assumptions of Lemma \ref{LCRem} are now satisfied for any choice of positive numbers $\beta_i$, $i\in \en$, such that $\sum \beta_i = \frac{1}{4}$.
\end{proof}

The following two lemmas are concerned with the Koch curve and constitute a partial result regarding its $c$-removability. See also Problem~\ref{KochProblem}. 

\begin{lemma}\label{limit}
$$
\lim_{k\to\infty}\left(3^k\prod\limits_{j=0}^{k-1}\frac{3^{j+1}+3}{3^{j+1}+1}
-2\sum\limits_{m=0}^{k-1}3^m\prod\limits_{j=0}^{m-1}\frac{3^{j+1}+3}{3^{j+1}+1}\right)=\infty\,.
$$
\end{lemma}

\begin{proof}
First, consider the following formula which follows easily by induction
$$
\prod\limits_{j=0}^{k-1}\frac{3^{j+1}+3}{3^{j+1}+1}=\frac{2\cdot 3^k}{3^k+1}\,.
$$
Now 
$$
\begin{aligned}
2\left(\frac{3^{2k}}{3^k+1}-2\sum\limits_{m=0}^{k-1}\frac{3^{2m}}{3^m+1}\right)
=&2\left(3^k-\frac{3^{k}}{3^k+1}-2\sum\limits_{m=0}^{k-1}\left(3^m-\frac{3^{m}}{3^m+1}\right)\right)\\
=&2\left(3^k-2\sum\limits_{m=0}^{k-1}3^m-\frac{3^{k}}{3^k+1}+2\sum\limits_{m=0}^{k-1}\frac{3^{m}}{3^m+1}\right)\\
\geq&2\left(3^k-2\frac{3^k-1}{3-1}-1+2\sum\limits_{m=0}^{k-1}\frac{3^{m}}{3^m+1}\right)\\
=&2\left(1-1+2\sum\limits_{m=0}^{k-1}\frac{3^{m}}{3^m+1}\right)\geq 4\sum\limits_{m=0}^{k-1}\frac{1}{2}=2k.
\end{aligned}
$$
\end{proof}

\begin{lemma}
Suppose that $f$ is a continuous function locally convex on the complement of the Koch curve.
Then $f$ is convex on every line parallel to the $y$-axis.
\end{lemma}

\begin{proof}
First note that it is sufficient to prove the statement of the lemma for a dense set of lines parallel to the $y$-axis.
Due to the self similarity of the Koch curve it is then sufficient to prove that there is no continuous function on $[0,3]\times[0,\frac{6}{\sqrt{3}}]$
such that $f(\frac{3}{2},0)+f(\frac{3}{2},\frac{6}{\sqrt{3}})<2f(\frac{3}{2},\frac{3}{\sqrt{3}}).$
Here we consider the realisation of the Koch curve with endpoints $(0,0)$ and $(3,0)$.
For simplicity we will work in coordinates where the point $(\frac{3}{2},\frac{3}{\sqrt{3}})$ is translated to the origin.
For $i\in\en_0$ denote
$$
a_{i}=\left(0,-\frac{\sqrt{3}}{3^{i+1}}\right),\quad b_{i}=\left(\frac{1}{2\cdot 3^i},-\frac{\sqrt{3}}{3^{i+1}}\right),
\quad u_{i}=\left(\frac{1}{2\cdot 3^i},-\frac{\sqrt{3}}{3^{i+1}}\right), 
$$
$$
z=\left(0,\frac{\sqrt{3}}{3}\right),\quad 
s_{i}=\left(\frac{1}{2\cdot 3^i},\frac{\sqrt{3}}{3}\right)\quad \text{and}\quad
p_{i}=\left(\frac{1}{2\cdot 3^i},0\right).
$$
Modifying $f$ by adding an appropriate affine function and multiplying it by an appropriate constant we can suppose that
$f(z)\leq0$, $f(s_i)\leq0$, $f(a_0)=1$ and $f(p_i)\geq f(0,0)=1.$
Since $f$ is convex on $[s_i,b_i]$ for every $i$ we can write
\begin{equation}\label{one}
\begin{aligned}
f(b_i)\geq& f(u_{i+1})+\frac{|b_i-u_{i+1}|}{|s_i-u_{i+1}|}(f(u_{i+1})-f(s_{i}))\geq
\left(1+\frac{\frac{\sqrt{3}}{3^{i+1}}-\frac{\sqrt{3}}{3^{i+2}}}{\frac{\sqrt{3}}{3}+\frac{\sqrt{3}}{3^{i+2}}}\right)f(u_{i+1})\\
=&\left(1+\frac{2}{3^{i+1}+1}\right)f(u_{i+1})=\frac{3^{i+1}+3}{3^{i+1}+1}f(u_{i+1}).
\end{aligned}
\end{equation}
Moreover, since $f$ is convex on $[a_i,u_i]$ for every $i$ we can write
\begin{equation}\label{two}
\begin{aligned}
f(u_i)\geq& f(b_{i})+\frac{|b_i-u_{i}|}{|b_i-a_{i}|}(f(b_{i})-f(a_{i}))\\
\geq&f(b_i)+\left(\frac{\frac{1}{2\cdot 3^i}-\frac{1}{2\cdot 3^{i+1}}}{\frac{1}{2\cdot 3^i}}\right)(f(b_{i})-1)=3f(b_i)-2.
\end{aligned}
\end{equation}
Combining \eqref{one} and \eqref{two} we then obtain
$$
f(b_i)\geq \frac{3^{i+1}+3}{3^{i+1}+1}(3f(b_i)-2)\,,
$$
and iterating for every $i<k$,
$$
f(b_i)\geq 3^k\prod_{j=0}^{k-1}\frac{3^{j+1}+3}{3^{j+1}+1}f(b_{i+k})
-2\sum\limits_{m=0}^{k-1}3^m\prod\limits_{j=0}^{m-1}\frac{3^{j+1}+3}{3^{j+1}+1}\,.
$$
Since $f$ is convex on $[s_i,b_i]$ we have for every $i$,
$$
f(b_i)\geq f(p_{i+1})+\frac{|s_i-p_{i+1}|}{|b_i-p_{i+1}|}(f(p_{i+1})-f(s_{i}))\geq f(p_{i+1})\geq 1.
$$
Finally, for $i=0$  and any $k>0$ we obtain
$$
f(b_0)\geq 3^k\prod\limits_{j=0}^{k-1}\frac{3^{j+1}+3}{3^{j+1}+1}
-2\sum\limits_{m=0}^{k-1}3^m\prod\limits_{j=0}^{m-1}\frac{3^{j+1}+3}{3^{j+1}+1}
$$
which is not possible due to Lemma~\ref{limit}
\end{proof}

\section{Open problems}

\noindent The following general question is likely to be very difficult to answer, but naturally arises from the introducion of the notion of $c$-removability.

\begin{problem} \label{ProbUltimatni}
Is there any interesting characterization of $c$-removable sets? 
\end{problem}

However, there are several other interesting problems whose solutions might contribute to our understanding of the matter.

\begin{problem}\label{NullProblem}
Is there a closed totally disconnected Lebesgue null set in $\er^2$ which is not $c$-removable?
\end{problem}

\begin{problem}
Is the Cantor dust $c$-removable?
\end{problem}

\begin{problem}\label{KochProblem}
Is there a non-trivial $c$-removable continuum in $\er^2$? 
\end{problem}

Note that if one could prove that there is a dense set of lines intersecting the Koch curve in countably many points, the answer would be positive; this would follow from the proof of Theorem~\ref{theorem1}.

\end{document}